\documentclass[a4paper,reqno,11pt]{amsart}
\usepackage{amsfonts,amssymb,amsmath,amsthm,abstract,color}
\usepackage[mathscr]{euscript}
\usepackage[ps,all,arc,rotate]{xy}
\usepackage[lmargin=1in,rmargin=1in,tmargin=1in,bmargin=1in]{geometry}
\usepackage{fancyhdr}
\usepackage{pb-diagram}
\usepackage{enumitem}

\usepackage{hyperref}
\usepackage{cleveref}

\usepackage{bbm}

\newtheorem{prop}{Proposition}[section]
\newtheorem{lemma}[prop]{Lemma}
\newtheorem{thm}[prop]{Theorem}

\theoremstyle{definition}
\newtheorem{defn}[prop]{Definition}

\newtheorem{rmk}[prop]{Remark}

\newtheorem{notn}[prop]{Notation}

\DeclareMathOperator{\Coh}{Coh} 
\newcommand{\TCoh}{\widetilde{\Coh}}
\DeclareMathOperator{\gr}{gr}

\DeclareMathOperator{\rk}{rk}        
 \DeclareMathOperator{\Sym}{Sym}

\DeclareMathOperator{\id}{id}

\DeclareMathOperator{\End}{End}

\newcommand{\ra}{\rightarrow}

\DeclareMathOperator{\Bun}{Bun}

\DeclareMathOperator*{\Lim}{Lim}
\DeclareMathOperator{\Lis}{Lis}
\DeclareMathOperator{\hLis}{hLis}

\DeclareMathOperator{\Jac}{Jac} 
\DeclareMathOperator{\Hom}{Hom}
\DeclareMathOperator{\Map}{Map}
\DeclareMathOperator{\Mor}{Mor}
\DeclareMathOperator{\Spec}{Spec}

\DeclareMathOperator{\Ext}{Ext}

\DeclareMathOperator{\Stk}{Stk}

\DeclareMathOperator{\Alg}{Alg}

\DeclareMathOperator{\Endspe}{end}
\newcommand{\Sp}{\mathrm{Sp}}
\newcommand{\Spc}{\mathrm{Spc}}
\newcommand{\Set}{\mathrm{Set}}
\DeclareMathOperator{\Pres}{Pr}
\DeclareMathOperator{\Cat}{Cat}

\DeclareMathOperator{\LMod}{LMod}

\def\cC{\mathcal C}\def\cD{\mathcal D}
\def\cE{\mathcal E}

\def\cM{\mathcal M}\def\cN{\mathcal N}

\def\cW{\mathcal W}
\def\cY{\mathcal Y}\def\cZ{\mathcal Z}

\def\AA{\mathbb A}\def\CC{\mathbb C}
\def\GG{\mathbb G}

\def\QQ{\mathbb Q}\def\SS{\mathbb S}

\def\ZZ{\mathbb Z}

\def\fX{\mathfrak X}

 \def\GL{\mathrm{GL}} \def\SL{\mathrm{SL}}
\def\PGL{\mathrm{PGL}}
\def\DM{\mathrm{DM}} \def\hDM{\mathrm{hDM}}     
\def\CH{\mathrm{CH}}  

\makeatletter
\def\@tocline#1#2#3#4#5#6#7{\relax
  \ifnum #1>\c@tocdepth 
  \else
    \par \addpenalty\@secpenalty\addvspace{#2}%
    \begingroup \hyphenpenalty\@M
    \@ifempty{#4}{%
      \@tempdima\csname r@tocindent\number#1\endcsname\relax
    }{%
      \@tempdima#4\relax
    }%
    \parindent\z@ \leftskip#3\relax \advance\leftskip\@tempdima\relax
    \rightskip\@pnumwidth plus4em \parfillskip-\@pnumwidth
    #5\leavevmode\hskip-\@tempdima
      \ifcase #1
       \or\or \hskip 1em \or \hskip 2em \else \hskip 3em \fi%
      #6\nobreak\relax
    \hfill\hbox to\@pnumwidth{\@tocpagenum{#7}}\par
    \nobreak
    \endgroup
  \fi}
\makeatother

\makeatletter
\newsavebox{\@brx}
\newcommand{\llangle}[1][]{\savebox{\@brx}{\(\m@th{#1\langle}\)}%
  \mathopen{\copy\@brx\kern-0.5\wd\@brx\usebox{\@brx}}}
\newcommand{\rrangle}[1][]{\savebox{\@brx}{\(\m@th{#1\rangle}\)}%
  \mathclose{\copy\@brx\kern-0.5\wd\@brx\usebox{\@brx}}}
\makeatother

\title[Motives of stacks of coherent sheaves]{Voevodsky motives of stacks of coherent sheaves on a curve}

\author{Victoria Hoskins and Simon Pepin Lehalleur}

\thanks{S. P. L. is supported by The Netherlands Organisation for Scientific Research (NWO), under project number 613.001.752.}

\begin{document}

\maketitle

\begin{abstract}
We prove formulae for the motives of stacks of coherent sheaves of fixed rank and degree over a smooth projective curve in Voevodsky's triangulated category of mixed motives with rational coefficients.
\end{abstract}

\setcounter{tocdepth}{1}
\tableofcontents

\section{Introduction}

Let $\Coh_{n,d}$ (resp. $\Bun_{n,d}$) denote the moduli stack of rank $n$, degree $d$ coherent sheaves (resp. vector bundles) on a smooth projective geometrically connected curve $C$ of genus $g$ over a field $k$. Cohomological invariants of the stack $\Bun_{n,d}$, as well as the related moduli spaces of (semi)stable vector bundles, have been intensely studied. One of the most classical results is due to Atiyah and Bott \cite{AB}: over $k = \CC$, the K\"{u}nneth components of the Chern classes of the universal bundle over $\Bun_{n,d} \times C$ freely generate the rational Betti cohomology ring of $\Bun_{n,d}$. In \cite{Heinloth_coh}, Heinloth showed the rational (Betti and $\ell$-adic) cohomology of the stack $\Coh_{n,d}$ with $n>0$ is also freely generated by the K\"{u}nneth components of the Chern classes of the universal family of coherent sheaves; the result turns out to be independent of the positive rank, which is possible because the universal sheaf has Chern classes in all cohomological degrees.

In previous work \cite{HPL_formula}, we computed the motive of $\Bun_{n,d}$ in Voevodsky's triangulated category $\DM(k,\QQ)$ of mixed motives over $k$ with $\QQ$-coefficients. In this paper, we build on that computation to describe the motive of $\Coh_{n,d}$.

\begin{thm}\label{main thm}
Suppose that $C(k) \neq \emptyset$; then in $\DM(k,\QQ)$, we have the following isomorphisms.
\begin{enumerate}
\item (Theorem \ref{thm torsion sheaves}) For the stack of torsion sheaves of degree $d$, we have
\[M(\Coh_{0,d}) \simeq  \Sym^d(M(\Coh_{0,1})) \simeq \Sym^d(M(C \times B\GG_m)). \]
\item \cite[Theorem 1.1]{HPL_formula} For the stack of vector bundles of rank $n >0$ and degree $d$, we have
\[ M(\Bun_{n,d}) \simeq M(\Jac(C)) \otimes M(B\GG_m) \otimes \bigotimes_{i=1}^{n-1} Z(C, \QQ\{i\}).\]
\item (Theorem \ref{thm coh positive rank}) For the stack of coherent sheaves of rank $n> 0$, we have
\[ M(\Coh_{n,d}) \simeq M(\Jac(C)) \otimes M(B\GG_m) \otimes \bigotimes_{i=1}^{\infty} Z(C, \QQ\{i\}).\]
\end{enumerate}
Here $Z(C,\QQ\{i\}):=\bigoplus_{j=0}^{\infty} M(\Sym^{j}(C)\otimes \QQ\{ij\}$ is a motivic Zeta function and $\QQ\{i\} := \QQ(i)[2i]$. 
\end{thm}

The infinite tensor product in the third statement should be interpreted as a colimit over $m\geq 1$ of the finite tensor products for $1\leq i\leq m$, with the transition maps given by the morphisms $\QQ\{0\}\to Z(C,\QQ\{m+1\})$ corresponding to the inclusion of the zeroth summand.

In particular, these decompositions imply corresponding decompositions on Chow groups and $\ell$-adic cohomology. Moreover, we note that all these motives are pure (in the sense of Definition \ref{def:pure}) and lie in the localising tensor subcategory of $\DM(k,\QQ)$ generated by the motive of the curve $C$. On the level of cohomology, these results were obtained by Laumon \cite{Laumon}, Atiyah-Bott \cite{AB} and Heinloth \cite{Heinloth_coh} and our arguments draw heavily on ideas from these papers. 

One of the key technical ingredients required is a motivic description of small maps of smooth stacks $f : \fX \ra \cY$ which are generically torsors under a finite group $G$; this requires us to work with motives with $\QQ$-coefficients in order to define the $G$-invariant piece. On the level of cohomology such results were used by Laumon \cite{Laumon} and Heinloth \cite{Heinloth_LaumonBDay}. Chow motives of small maps were investigated by de Cataldo and Migliorini \cite{dCM_small} and we used their work to describe Voevodsky motives of small maps of varieties (see \cite[Theorem 2.13]{HPL_formula}) in our description of the motive of $\Bun_{n,d}$. In this paper, we extend \cite[Theorem 2.13]{HPL_formula} to stacks (see Theorem \ref{thm short small maps stacks}). The main challenge is to show that the $G$-action coming from the description as a torsor on a dense open extends to the motive of $\fX$. For this, we work with relative motives over stacks and adopt an $\infty$-categorical approach to categories of motivic sheaves over stacks similarly to \cite{khan-stacks}. Fortunately to define an action on the relative motive $f_!f^!\QQ_{\cY}$ over $\cY$, we can avoid discussing homotopically coherent group actions, as the endomorphism space of this motive turns out to be discrete. 

We apply Theorem \ref{thm short small maps stacks} in the proof of Theorem \ref{thm torsion sheaves} and also in our recent paper \cite{HPL_mirrorsym} on motivic mirror symmetry for Higgs bundles for dual Langlands groups $\SL_n$ and $\PGL_n$, in order to prove the motive of the moduli space of semistable tracefree Higgs bundles of fixed determinant of degree coprime to the rank is abelian (see \cite[Theorem 1.2]{HPL_mirrorsym}). 

Moduli stacks of vector bundles and coherent sheaves play an important role in the geometric Langlands correspondence for $\GL_{n}$ \cite{Laumon} and (cohomological) Hall algebras of curves \cite{Schiffmann_survey,SSala,SV_preprint}. In future work, we plan to use Theorem \ref{main thm} to construct and study Voevodsky motivic Hall algebras for coherent sheaves on curves, building on the cohomological version in \cite{SSala,SV_preprint}.


\subsection*{Conventions and motivic set-up}

\subsubsection*{Stacks}
By a stack, we mean an algebraic stack which is locally of finite type over a field $k$.

\subsubsection*{$\infty$-categories}

An $\infty$-category is a quasicategory as in \cite{HTT,HA}. We let $\Cat_{\infty}$ denote the $\infty$-category of small $\infty$-categories, $\widehat{\Cat}_{\infty}$ denote the $\infty$-category of large $\infty$-categories, and $\Pres^{L}$ denote the $\infty$-category of presentable $\infty$-categories and left-exact functors. We let $\Spc$ denote the $\infty$-category of spaces (i.e. $\infty$-groupoids) and let $\Sp$ denote the $\infty$-category of spectra. 

\subsubsection*{Categories of motives}

We fix a base field $k$ and as we are working with $\QQ$-coefficients, we work with categories of \'{e}tale motivic sheaves. For this paper, we need to work not only with motives \emph{of} algebraic stacks in $\DM(k,\QQ)$, but also with motivic sheaves \emph{on} algebraic stacks, i.e. with categories $\DM(\fX,\QQ)$ where $\fX$ is an algebraic stack (locally of finite type over the base field $k$). For this purpose, we adopt the formalism of \cite[Appendix A]{khan-stacks}, adapted to the setting of $\DM(-,\QQ)$ (see \cite[Appendix A.1]{HPL_mirrorsym} for a summary of this set-up in the context of $\DM(-,\QQ)$). This requires working systematically with stable $\infty$-categories rather than triangulated categories. We denote the triangulated homotopy category of $\DM(-,\QQ)$ by $\hDM(-,\QQ)$. For a stack $\cY$, we will repeatedly use the description of $\DM(\cY,\QQ)$ given in \cite[Eq.\ (A.4)]{khan-stacks}:
\[
\DM(\cY,\QQ)\simeq \Lim_{u:Y\to \cY\in \Lis_{\cY}} \DM(Y,\QQ)
\]
where $\Lis_{\cY}$ is the full $2$-subcategory of the category $\Stk_{\cY}$ of stacks over $\cY$ consisting of objects which are smooth and quasi-projective over $k$, the transition maps in this limit are given by ordinary pullbacks and the limit is taken in $\Pres^{L}$. The fact that the limit can be taken with only quasi-projective varieties is not stated in loc. cit. but follows from Zariski descent.

There is a full six-operation formalism for $\DM(-,\QQ)$ with respect to morphisms of algebraic stacks \cite{khan-stacks}. In particular, if $f:\fX\to \cY$ is a locally of finite type morphism between algebraic stacks, there is an adjunction
\[
f_{!}:\DM(\fX,\QQ)\leftrightarrows \DM(\cY,\QQ):f^{!}.
\]
If $\fX$ is locally of finite type over $k$ with structure morphism $\pi_{\fX}:\fX\to \Spec(k)$, we can thus define the motive of $\fX$ as
\[
M(\fX):=\pi_{\fX!}\pi_{\fX}^{!}\QQ_{\fX}.
\]
Since $\DM(\fX,\QQ)$ is defined by an étale descent procedure, this motive can also be written as an homotopy colimit of the motive of the simplicial algebraic space attached to an atlas; if $\fX$ admits is smooth and admits an atlas which is a scheme, we see that we recover up to isomorphism the construction of étale motives of stacks used in our previous papers \cite[Appendix A]{HPL} and \cite{HPL_formula}.

\section{Small maps between stacks and motives}\label{sec small maps}

In this section, we extend a result on motives of small maps given in \cite[Theorem 2.13]{HPL_formula} based on \cite{dCM_small} from schemes to stacks; for the definition of a small map of stacks, see \cite[Definition 2.4]{HPL_formula}. Our aim is to prove the following result, which is Theorem \ref{thm small maps stacks} \emph{\ref{iso-base}} below.

\begin{thm}\label{thm short small maps stacks}
 Let $f : \fX \ra \cY$ be a small proper surjective representable morphism of smooth stacks such that there exists a dense open $\cY^\circ \subset \cY$ with preimage $\fX^\circ$ such that $f^{\circ}: \fX^\circ \ra \cY^\circ$ is a $G$-torsor. Then the $G$-action on $M(\fX^\circ)$ extends to a $G$-action on $M(\fX)\in \hDM(k,\QQ)$ and there is an isomorphism
  \[ M(\fX)^G \simeq M(\cY).\]
\end{thm}

In the scheme case, the key ingredient in the proof was \cite[Proposition 2.10]{HPL_formula}: for a small proper surjective morphism $f:X\to Y$ of smooth schemes with base change $f^{\circ}:X^{\circ}\to Y^{\circ}$ to a dense open subscheme $j : Y^{\circ} \hookrightarrow Y$, the morphism
\[
j^{*}:\End_{\hDM(Y)}(f_{!}f^{!}\QQ_{Y})\to \End_{\hDM(Y^{\circ})}(f^{\circ}_{!}f^{\circ!}\QQ_{Y^{\circ}})
\]
is a bijection of sets. When $f^{\circ}$ is a $G$-torsor, this allows us to extend the $G$-action on $f^{\circ}_{!}f^{\circ!}\QQ_{Y^{\circ}}$ to $f_{!}f^{!}\QQ_{Y}$ and to ultimately cut out the motive of $Y$ as a direct summand of the motive of $X$.

To prove the same result for stacks, we proceed by descent from the scheme case. A priori, this requires studying the full endomorphism mapping space $\End_{\DM(Y)}(f_{!}f^{!}\QQ_{Y})\in \Spc$ in the $\infty$-category $\DM(Y,\QQ)$, and to contend with some homotopy coherence problems for group actions in $\infty$-categories. However, it turns out that this mapping space is actually discrete\footnote{That is, all its higher homotopy groups are trivial.}, which allows us to sidestep these issues.

\subsection{Endomorphisms objects in (stable) $\infty$-categories}

Let $\cC^{\otimes}$ be a monoidal $\infty$-category and $\cM$ be an $\infty$-category enriched over $\cC^{\otimes}$ in the sense of \cite[Definition 4.2.1.28]{HA}. If $\cC^{\otimes}=\Spc^{\times}$ is the $\infty$-category of spaces with its cartesian monoidal structure, then every presentable $\infty$-category is canonically enriched over $\Spc^{\times}$ by \cite[\S 4.4.4]{HTT} and \cite[Proposition 4.2.1.33]{HA}. Similarly, if $\cC^{\otimes}=\Sp^{\otimes}$ is the $\infty$-category of spectra equipped with the smash product, then every presentable stable $\infty$-category is canonically enriched over $\Sp^{\otimes}$ by \cite[Remark 4.8.2.20, Proposition 4.2.1.33]{HA}. These are the only examples we will need in this paper.

By definition, this set-up implies that $\cM$ is left-tensored over $\cC^{\otimes}$, so that one has a functor
\begin{equation}\label{action functor}\tag{$\star$}
\otimes:\cC^{\otimes}\times \cM\to \cM
\end{equation}
and that given two objects $M,M'\in \cM$, there is a \emph{morphism object}
\[
\Mor_{\cM}^{\cC^\otimes}(M,M')\in\cC^{\otimes}
\]
together with a morphism $\alpha:\Mor_{\cM}^{\cC^\otimes}(M,M')\otimes M\to M'$ in $\cM$ with the universal property that for any object $C\in\cC^{\otimes}$, the composition
\[
\Map_{\cC}(C,\Mor^{\cC^\otimes}_{\cM}(M,M'))\stackrel{-\otimes \id_{M}}{\to}\Map_{\cM}(C\otimes M,\Mor^{\cC^{\otimes}}_{\cM}(M,M')\otimes M)\stackrel{\alpha^{*}}{\to}\Map_{\cM}(C\otimes M,M')
\]
is an isomorphism in $\Spc$. When $M=M'$, we write
\[
\End_{\cM}^{\cC^\otimes}(M):=\Mor_{\cM}^{\cC^\otimes}(M,M)\in \cC^{\otimes}.
\]

\begin{rmk}
By analogy with $1$-category theory, one expects that composition of morphisms induces an associative algebra structure on $\End_{\cM}^{\cC^\otimes}(M)$, i.e. an object in $\Alg(\cC^{\otimes})$. It should also be the case that $\End_{\cM}^{\cC^\otimes}(M)$ acts on $M$ and is in fact the universal object in $\Alg(\cC^{\otimes})$ equipped with an action on $M$, in the sense that for any object $A\in \Alg(\cC^{\otimes})$, there is an isomorphism
\[
\Map_{\Alg(\cC^{\otimes})}(A,\End_{\cM}(M))\simeq {A}\times_{\Alg(\cC^{\otimes})}\LMod(\cM)\times_{\cM}{M}
\]
in $\Spc$, where $\LMod(\cM)$ is the $\infty$-category whose objects are triples $(B,M',\bullet)$ with $B\in \Alg(\cC^{\otimes})$, $M'\in\cM$ and $\bullet :B\otimes M'\to M'$ an action of $B$ on $M'$. The above sketch is realised and made precise and functorial in \cite[\S 4.7.1 and \S 4.8.5]{HA} in the presentable setting. Moreover, the discussion on functoriality and limits below are compatible with these algebra structures. However we will not need to know about these homotopy-coherent compositions because in our application the endomorphism spaces turn out to be discrete, and we are reduced to usual monoids in $\Set$.
\end{rmk}

\begin{notn}\
\begin{enumerate}
\item For an object $M$ in a presentable $\infty$-category $\cM$ (which is canonically left-tensored over $\Spc^{\times}$), the above construction produces an endomorphism space denoted
\[
\End_{\cM}(M):=\End^{\Spc^{\times}}_{\cM}(M)\in\Spc.
\]
\item For an object $M$ in a presentable stable $\infty$-category $\cM$ (which is canonically left-tensored over $\Sp^{\otimes}$), the above construction produces an endomorphism spectrum denoted
\[
\Endspe_{\cM}(M):=\End_{\cM}^{\Sp^{\otimes}}(M)\in \Sp.
\]
\end{enumerate}
\end{notn}

\begin{lemma}\label{lemma:pi-endspe}
For an object $M$ in a presentable stable $\infty$-category $\cM$ and $n\in \ZZ$,
    \[
\pi_{n}\Endspe_{\cM}(M)\simeq \Ext^{-n}_{h\cM}(M,M)
\]
can be computed as an Ext-group in the triangulated homotopy category $h\cM$ of $\cM$.
\end{lemma}
\begin{proof}
We have isomorphisms
  \[
\pi_{n}\Endspe_{\cM}(M)\simeq \Map_{\Sp}(\Sigma^{n}\SS,\Endspe_{\cM}(M))\simeq \Map_{\cM}(\Sigma^{n}\SS\otimes M,M)\simeq \Map_{h\cM}(M[n],M)
 \]
using the universal property of morphism objects.
\end{proof}  

\begin{rmk} One can also show that the ring structure on $\pi_{*}\Endspe_{\cM}(M)$ induced by composition is given by composition in the triangulated category $h\cM$ (see \cite[Remark 7.1.2.2]{HA} for a discussion of this point), but we will not need this.
\end{rmk}

Let us explain how the endomorphism space and spectra are related by the infinite suspension functor. In general, let $F:\cC^{\otimes}\to \cD^{\otimes}$ be a (strong) monoidal functor between presentable monoidal $\infty$-categories, $\cM$ be a presentable $\infty$-category enriched over both $\cC^{\otimes}$ and $\cD^{\otimes}$ compatibly with $F$, and $M\in\cM$ be an object. Using the universal property of endomorphism objects, we obtain an induced morphism
\[
F_{*}:\End_{\cM}^{\cC^{\otimes}}(M)\to F(\End_{\cM}^{\cD^{\otimes}}(M))
\]
in $\cD^{\otimes}$. We now apply this to the monoidal functor of infinite suspension \cite[Example 6.2.4.13]{HA}
\[
F=\Sigma^{\infty}_{+}:\Spc^{\times}\to \Sp^{\otimes}.
\]
Let $\cM$ be a stable presentable $\infty$-category. By the above discussion, $\cM$ is canonically enriched over both $\Spc^{\times}$ and $\Sp^{\otimes}$, and the two enrichments are compatible because the actions \eqref{action functor} are (this is essentially contained in \cite[Proposition 4.8.2.18]{HA}). Hence, for any object $M\in \cM$, we get a morphism
\[
(\Sigma^{\infty}_{+})_{*}:\Sigma^{\infty}_{+}\End_{\cM}(M)\to \Endspe_{\cM}(M)
\]
in $\Sp$. The functor $\Sigma^{\infty}_{+}$ has a right adjoint $\Omega^{\infty}:\Sp\to \Spc$, so we get a morphism
\[
\Omega^{\infty}:\End_{\cM}(M)\to \Omega^{\infty}\Endspe_{\cM}(M)
\]
in $\Spc$.
\begin{lemma}\label{lemma:omega-infty}
With the above notation, the morphism
  \[
\Omega^{\infty}:\End_{\cM}(M)\to \Omega^{\infty}\Endspe_{\cM}(M)
\]
is an isomorphism, and for all $n\geq 0$, we have
\[
\pi_{n}\End_{\cM}(M)\simeq \Ext^{-n}_{h\cM}(M,M).
\]
\end{lemma}  
\begin{proof}
By the Yoneda Lemma and the universal property of morphism objects above, for any $K\in \Spc$, we have
\[
\Map_{\Spc}(K,\End_{\cM}(M))\simeq \Map_{\cM}(K\otimes M,M)
\]
and
\[
\Map_{\Spc}(K,\Omega^{\infty}\Endspe_{\cM}(M))\simeq \Map_{\Sp}(\Sigma^{\infty}_{+}K,\Endspe_{\cM}(M))\simeq \Map_{\cM}(\Sigma^{\infty}_{+}K\otimes M,M).
\]
As we already mentioned before, the two actions are compatible: $\Sigma^{\infty}_{+}K\otimes M\simeq K\otimes M$ functorially in $K$, and this concludes the proof of the first statement.

The second statement follows from the first together with Lemma \ref{lemma:pi-endspe}.
\end{proof}  

Finally, we describe how these endomorphism objects interact with functors and with limits of diagrams of $\infty$-categories.

Let $F:\cM\to \cN$ be a left-adjoint functor between presentable $\infty$-categories and $M\in \cM$ be an object. Then by the universal property, we get an induced morphism
\[
F_{*}:\End_{\cM}(M)\to \End_{\cM}(F(M))
\]
in $\Spc$.

\begin{lemma}\label{lemma lim end}
Let $I$ be a small $\infty$-category and $\cM_{*}:I\to\Pres^{L}$ be a diagram of presentable $\infty$-categories and left-adjoint functors. For an object $M$ in $\Lim_{i\in I}\cM_{i}$ with projections $M_{i}\in \cM_{i}$, there is an isomorphism
  \[
\End_{\Lim_{i\in I}\cM_{i}}(M)\simeq \Lim_{i\in I}\End_{\cM_{i}}(M_{i})
\]
in $\Spc$.
\end{lemma}  
\begin{proof}
By \cite[Proposition 5.5.3.13]{HTT}, the limit of $\cM_{*}$ is the same whether taken in $\Pr^{L}$ or in $\widehat{\Cat}_{\infty}$. Thus it suffices to prove the result in $\widehat{\Cat}_{\infty}$, which is then a standard result about limits of $\infty$-categories. For instance, using the equivalence between the quasicategory model and the simplicial category model of $\infty$-categories, it can be deduced from \cite[Proposition A.3.2.27]{HTT}. 
\end{proof}

\subsection{Motives of small maps of algebraic stacks}

In order to prove that the relevant endomorphism space is discrete, we first make the following observation about smooth base change for relative homology.

\begin{lemma}\label{lemma:base-change}
  Let $f:\fX\ra \cY$ be a separated finite type morphism of algebraic stacks. For a smooth morphism $u:\cZ\to \cY$ of stacks, consider the base change
\[ \xymatrix{ \fX \times_{\cY} \cZ \ar[r]^{\tilde{u}} \ar[d]_{f_{\cZ}} & \fX \ar[d]^{f} \\ \cZ \ar[r]^{u} & \cY. }\] 
\begin{enumerate}[label=\emph{\roman*)}]
\item There is an isomorphism
  \[
i_{u}:u^{*}f_{!}f^{!}\QQ_{\cY}\simeq f_{Z!}f_{Z}^{!}\QQ_{\cZ}.
\]
\item The isomorphism in \emph{i)} is ($1$-categorically) functorial in the following sense: given another smooth morphism $v:\cW\to \cZ$, we have $i_{v} \circ v^*i_{u} = i_{u \circ v}$ in $\hDM(\cW,\QQ)$.
\end{enumerate}
\end{lemma}
\begin{proof}
  By working componentwise on $\cZ$, we can assume that $u$ has a fixed relative dimension $d$. For i), we have
\[  u^{*}f_{!}f^{!}\QQ_{\cY}\simeq  (f_{Z})_{!}\tilde{u}^{*}f^{!}\QQ_{\cY} \simeq  (f_{Z})_{!}\tilde{u}^{!}f^{!}\QQ_{\cY}\{-d\} \simeq  (f_{Z})_{!}f_{Z}^{!}u^{!}\QQ_{\cY}\{-d\} \simeq  (f_{Z})_{!}f_{Z}^{!}\QQ_{\cY}\]  
by base change and relative purity \cite[Appendix A]{khan-stacks}. The functoriality stated in ii) is an easy consequence of this construction.
\end{proof}  

We now present the key result which makes the situation manageable.

\begin{prop}\label{prop:end-discrete}
  Let $f : \fX \ra \cY$ be a small proper surjective representable morphism of smooth stacks. Then the space $\End_{\DM(\cY)}(f_{!}f^{!}\QQ_{\cY})$ is discrete.
\end{prop}  
\begin{proof}
  We first prove the result in the case where $\cY=Y$ is a smooth quasi-projective variety over $k$. In other words, it suffices to show that for all small proper surjective morphism $f:X\to Y$ between smooth varieties over $k$ with $Y$ quasi-projective and all $n>0$, we have
  \[
\pi_n\End_{\DM(Y,\QQ)}(f_{!}f^{!}\QQ_{Y})=0.
\]
By Lemma \ref{lemma:omega-infty}, this is equivalent to showing that
\[
\Hom_{\DM(Y,\QQ)}(f_{!}f^{!}\QQ_{Y},f_{!}f^{!}\QQ_{Y}[-n])=0.
\]
To prove this, we can apply Zariski descent and Lemma \ref{lemma:base-change} i) to the open inclusion of connected components to reduce to the case where $Y$ is connected of dimension $d$. By \cite[Lemma 2.7]{HPL_formula}, the (possibly singular) variety $X\times_YX$ is equidimensional of dimension $d$ . By the same computation as in the proof of \cite[Proposition 2.10]{HPL_formula}, we have
\[
\Hom_{\DM(Y,\QQ)}(f_{!}f^{!}\QQ_{Y},f_{!}f^{!}\QQ_{Y}[-n])\simeq \Hom_{\DM(k,\QQ)}(\QQ(d)[2d+n],M_c(X\times_YX)).
\]
By \cite[Corollary 8.12]{CD_int}, we have
\[
\Hom_{\DM(k,\QQ)}(\QQ(d)[2d+n],M_c(X\times_YX))\simeq \CH^0(X\times_YX,n)\otimes\QQ,
\]
where $\CH^0(X\times_YX,n)$ denotes Bloch's higher Chow group of cycles of codimension $0$ and weight $n$. By definition, $\CH^0(X\times_YX,n)=H_n(z^0(X\times_YX,*))$ where $z^0(X\times_YX,i)$ is the free abelian group generated by the codimension $0$ integral subvarieties of $(X\times_YX)\times_k\Delta^i$ \cite[Definition 17.1]{MVW}, with the additional condition of proper intersection being automatic in this case. Since the codimension $0$ integral subvarieties of $(X\times_YX)\times_k\Delta^i$ are just the fibre products of the irreducible components of $X\times_YX$ with $\Delta^i$, we see that we indeed have $\CH^0(X\times_YX,n)=0$ for $n>0$ (and $\CH^0(X\times_YX,0)=Z_d(X\times_YX)=\CH_d(X\times_YX)$). This finishes the proof in this case.

We now turn to the general case. As in \cite[Appendix A]{khan-stacks}, we have 
\[
\DM(\cY,\QQ)\simeq \Lim_{u:Y\to \cY\in \Lis_{\cY}} \DM(Y,\QQ).
\]
Hence by Lemma \ref{lemma lim end}, we deduce
that
\[
  \End_{\DM(\cY,\QQ)}(f_{!}f^{!}\QQ_{\cY}) \simeq  \Lim_{u:Y\to \cY\in \Lis_{\cY}} \End_{\DM(Y,\QQ)}(u^{*}f_{!}f^{!}\QQ_{\cY}) 
\]
in $\Spc$. By Lemma \ref{lemma:base-change} i), we have $u^{*}f_{!}f^{!}\QQ_{\cY}\simeq f_{Y_!}f_Y^{!}\QQ_Y$ where $f_Y$ is the base change of $f$ along $u : Y \ra \cY$. From the case of quasi-projective smooth varieties above, we know that $\End_{\DM(Y,\QQ)}(f_{Y_!}f_Y^{!}\QQ_Y)$ is discrete. A limit of discrete spaces is discrete (as $\pi_{0}:\Spc\to \Set$ is the left adjoint of the inclusion $\Set\to \Spc$), so we conclude that $\End_{\DM(\cY,\QQ)}(f_{!}f^{!}\QQ_{\cY})$ is discrete.
\end{proof}  

This discreteness enables us to use the following standard lemma comparing $\infty$-categorical and $1$-categorical limits.

\begin{lemma}\label{lemma:lim-discrete}
  Let $I$ be a small $\infty$-category and let $F:I\to\Spc$ be a diagram of spaces. Assume that for all $i\in I$, the space $F(i)$ is discrete. Then $\Lim F$ is discrete and its $\pi_{0}$ can be described as follows: the functor $\pi_{0}F:I\to \Set$ factors essentially uniquely through the homotopy category $hI$, giving a functor $\bar{F}:hI\to \Set$, and we have
  \[
\pi_{0}\Lim F\simeq \Lim \bar{F}
\]
in $\Set$.
\end{lemma}
\begin{proof}
This is a special case of \cite[\href{https://kerodon.net/tag/02JD}{Tag 02JD}]{kerodon}.
\end{proof}  

We can now prove an extension of \cite[Proposition 2.10]{HPL_formula} to stacks.

\begin{prop}\label{prop restr iso}
  Let $f : \fX \ra \cY$ be a small proper surjective representable morphism of smooth stacks. Let $j : \cY^\circ \hookrightarrow \cY$ be the inclusion of a dense open subset and  $f^{\circ}:\fX^\circ\to\cY^\circ$ denote the base change of $f$ via $j$. Then the natural map of mapping spaces
  \[
j^{*}:\End_{\DM(\cY)}(f_{!}f^{!}\QQ_{\cY})\to \End_{\DM(\cY^{\circ})}(j^{*}f_{!}f^{!}\QQ_{\cY})
\]
is an isomorphism of discrete objects in $\Spc$.
\end{prop}
\begin{proof}

  We first prove the result in the case where $\cY=Y$ is a smooth quasi-projective variety over $k$. By Lemma \ref{lemma:base-change} i), it is enough to show the same statement for the morphism
  \[
j^{*}:\End_{\DM(Y)}(f_{!}f^{!}\QQ_{Y})\to \End_{\DM(Y^{\circ})}(f^{\circ}_{!}f^{\circ !}\QQ_{Y^{\circ}})
\]
induced by $j$ and the functoriality of relative homological motives. In \cite[Proposition 2.10]{HPL_formula}, we proved this result at the level of $\pi_0$. To complete the proof, it suffices to show that the source and the target of $j^*$ are discrete, which follows from Proposition \ref{prop:end-discrete}.
  
  We now turn to the general case. We again know that the source and target are discrete by Proposition \ref{prop:end-discrete}. As in \cite[Appendix A]{khan-stacks}, we have 
\[
\DM(\cY,\QQ)\simeq \Lim_{u:Y\to \cY\in \Lis_{\cY}} \DM(Y,\QQ)
\]
where the transition maps are given by ordinary pullbacks. We claim that the functor $j^{*}$ is an isomorphism because it factors as the following sequence of isomorphisms
\begin{eqnarray}
  \End_{\DM(\cY,\QQ)}(f_{!}f^{!}\QQ_{\cY}) &\simeq & \Lim_{u:Y\to \cY\in \Lis_{\cY}} \End_{\DM(Y,\QQ)}(u^{*}f_{!}f^{!}\QQ_{\cY})\\
     &\simeq& \Lim_{u:Y\to \cY\in \hLis_{\cY}}\End_{\hDM(Y,\QQ)}(u^{*}f_{!}f^{!}\QQ_{\cY})\\
  &\simeq& \Lim_{u:Y\to \cY\in \hLis_{\cY}}\End_{\hDM(Y^{\circ},\QQ)}(j_{Y}^{*}u^{*}f_{!}f^{!}\QQ_{\cY})\\
                                     &\simeq & \Lim_{u:Y\to \cY\in \hLis_{\cY}} \End_{\hDM(Y^{\circ},\QQ)}((u^{\circ})^{*}j^{*}f_{!}f^{!}\QQ_{\cY})\\
&\simeq & \Lim_{v:Y^{\circ}\to \cY^{\circ}\in \hLis_{\cY^{\circ}}}\End_{\hDM(Y^{\circ},\QQ)}(v^{*}j^{*}f_{!}f^{!}\QQ_{\cY})\\
&\simeq & \Lim_{v:Y^{\circ}\to \cY^{\circ}\in \Lis_{\cY^{\circ}}}\End_{\DM(Y^{\circ},\QQ)}(v^{*}j^{*}f_{!}f^{!}\QQ_{\cY})\\
  & \simeq & \End_{\DM(\cY^{\circ},\QQ)}(j^{*}f_{!}f^{!}\QQ_{\cY}).
\end{eqnarray}  
Here, Isomorphisms $(1)$ and $(7)$ follow from Lemma \ref{lemma lim end}, Isomorphisms $(2)$ and $(6)$ follow from Lemma \ref{lemma:lim-discrete} and Isomorphism $(3)$ follows from Lemma \ref{lemma:base-change} i) and the case of quasi-projective smooth schemes proved above. Isomorphism $(4)$ follows from the ($1$-categorical) functoriality of pullbacks. Finally, isomorphism $(5)$ follows from the fact that the base change functor
\[
-\times_{\cY}\cY^{\circ}:\hLis_{\cY}\to \hLis_{\cY^{\circ}}
\]
has a left adjoint (given by post-composition by $j$) and thus is initial (in the sense that postcomposing with it does not change the limit of a diagram).

It remains to check that the composition of these isomorphisms coincides with $j^{*}$. Since these are discrete spaces, it suffices to check this for elements, and it then follows easily from the universal properties of endomorphism objects and limits.
\end{proof}

We now have the ingredients to prove the main result of this section.

\begin{thm}\label{thm small maps stacks}
  Let $f : \fX \ra \cY$ be a small proper surjective representable morphism of smooth stacks such that there exists a dense open $\cY^\circ \subset \cY$ with preimage $\fX^\circ$ such that $f^{\circ}: \fX^\circ \ra \cY^\circ$ is a $G$-torsor. Then the following statements hold.
\begin{enumerate}[label=\emph{(\roman*)}]
\item\label{action} The $G$-action on $f^{\circ}_{!}f^{\circ!}\QQ$ in $\hDM(\cY^{\circ},\QQ)$ coming from the $G$-action on $\fX$ extends to an action of $G$ on $f_{!}f^{!}\QQ$ in $\hDM(\cY,\QQ)$, i.e. a morphism
  \[
\alpha_{\fX}:G\to \End_{\hDM(\cY,\QQ)}(f_{!}f^{!}\QQ_{\cY})
\]
of monoids in $\Set$.
\item\label{action-bc} This $G$-action is compatible with smooth base change in the following sense: if $u:\cZ\to\cY$ is a smooth morphism of smooth algebraic stacks, then the morphisms $\alpha_{\fX\times_{\cY}\cZ}$ coincides with the composition
  \[
\quad \quad \quad G\stackrel{\alpha_{\fX}}\longrightarrow \End_{\hDM(\cY,\QQ)}(f_{!}f^{!}\QQ_{\cY})\to \End_{\hDM(\cZ,\QQ)}(u^{*}f_{!}f^{!}\QQ_{\cY})\stackrel{i_u}{\simeq} \End_{\hDM(\cZ,\QQ)}(f_{\cZ_{!}}f_{\cZ}^{!}\QQ_{\cZ}),
  \]
where the isomorphism $i_u$ is the one given by Lemma \ref{lemma:base-change} \emph{i)}.
\item\label{iso-up} The following composition is an isomorphism
  \[
(f_{!}f^{!}\QQ_{\cY})^{G}\to f_{!}f^{!}\QQ_{\cY}\to \QQ_{\cY},
\]
where the first map is induced by the $G$-action $\alpha_{\fX}$ in the $\QQ$-linear category $\hDM(\cY,\QQ)$ and the second map is the counit of the adjunction.
\item\label{iso-base} There is an induced $G$-action on $M(\fX)\in \hDM(k,\QQ)$ and an isomorphism
  \[ M(\fX)^G \simeq M(\cY).\]
\end{enumerate}    
\end{thm}
\begin{proof}
  We start with the construction of the action in \emph{\ref{action}}. Since $G$ acts on $\fX^{\circ}$ over $\cY^{\circ}$, the functor
  \[
M_{\cY^{\circ}}:\Stk_{\cY^{\circ}}\to \hDM(\cY^{\circ},\QQ),\quad (h:\cY'\to \cY^{\circ})\mapsto h_{!}h^{!}\QQ_{\cY^\circ}
\]
induces an action of $G$ on $f^{\circ}_{!}f^{\circ!}\QQ_{\cY^\circ}$, thus a morphism
\[
\alpha':G\to \End_{\hDM(\cY^{\circ},\QQ)}(f^{\circ}_{!}f^{\circ!}\QQ_{\cY^\circ}).
\]
By Lemma \ref{lemma:base-change} i) and Proposition \ref{prop restr iso}, we have an isomorphism
\[
j^{*}:\End_{\hDM(\cY)}(f_{!}f^{!}\QQ_{\cY})\simeq \End_{\hDM(\cY^{\circ})}(f^{\circ}_{!}f^{\circ!}\QQ_{\cY^{\circ}})
\]
of sets and we define $\alpha_{\fX}$ as the composition
\[
G\stackrel{\alpha}{\to}\End_{\hDM(\cY^{\circ})}(f^{\circ}_{!}f^{\circ!}\QQ_{\cY^{\circ}})\stackrel{(j^{*})^{-1}}{\to}\End_{\hDM(\cY)}(f_{!}f^{!}\QQ_{\cY}).
\]
The fact that this construction is compatible with smooth base changes as claimed in \emph{\ref{action-bc}} then follows from the $1$-categorical functoriality of the isomorphism in Lemma \ref{lemma:base-change} ii), as we have 
\[i_u \circ (u^{\circ})^*i_j = i_{j \circ u^\circ} = i_{u \circ j_\cZ} = i_{j_\cZ} \circ (j_\cZ)^*i_u. \]

We now come to the main statement: part \emph{\ref{iso-up}}. By \cite[Eq.\ (A.4)]{khan-stacks}, the family of functors $u^{*}:\DM(\cY,\QQ)\to \DM(Y,\QQ)$ with $u:Y\to \cY$ varying across smooth morphisms from a smooth scheme $Y$ is jointly conservative. By Zariski descent, we can further assume that $Y$ is connected, separated and of finite type over $k$. By \emph{\ref{action-bc}}, this reduces \emph{\ref{iso-up}} to the case $\cY=Y$, which was proven in \cite[Theorem 2.13]{HPL_formula} (or rather, the proof in loc.\ cit.\ establishes \emph{\ref{iso-up}} in this case, even though the statement there was only about the analogue of \emph{\ref{iso-base}}).

Then \emph{\ref{iso-base}} follows by pushing forward \emph{\ref{iso-up}} to $\Spec(k)$, as in the proof of \cite[Theorem 2.13]{HPL_formula}.
\end{proof}

\section{The stack of torsion sheaves}

Following Laumon \cite[Section 3]{Laumon}, we consider the stack of full flags of torsion sheaves on $C$
\[ \TCoh_{0,d} := \langle 0 = T_0 \subset T_1 \subset \cdots \subset T_d = T :  T_i \in \Coh_{0,i} \rangle \]
which admits a forgetful map $\TCoh_{0,d}  \ra \Coh_{0,d}$ that Laumon showed is small and an $S_d$-torsor over the dense open set where the support of the torsion sheaf consists of $d$ distinct points \cite[Paragraph (3.2) and Theorem 3.3.1]{Laumon}. In fact, Laumon observed that if $C = \AA^1$, then this morphism coincides with the $\GL_d$-stack quotient of the Grothendieck-Springer resolution.

Consequently, on the level of (Betti or $\ell$-adic) rational cohomology, one has
\[ H^*(\Coh_{0,d}) \simeq H^*(\TCoh_{0,d})^{S_d}.\]
Since the associated graded map $\gr : \TCoh_{0,d}  \ra \prod_{i=1}^d \Coh_{0,1}$ is an iteration of vector bundle stacks (see Definition \ref{def vector bundle stack} below and \cite{Heinloth_coh}), whose fibres are contractible, one obtains
\[ H^*(\Coh_{0,d}) \simeq H^*(\TCoh_{0,d})^{S_d} \simeq H^*(\Pi_{i=1}^d \Coh_{0,1})^{S_d} \simeq \Sym^d(H^*(\Coh_{0,1}))\]
and, as the support map $\Coh_{0,1} \ra C$ is a $\GG_m$-gerbe (in fact, a trivial $\GG_m$-gerbe, as $C(k) \neq \emptyset$), one has $H^*(\Coh_{0,1}) \simeq H^*(C) \otimes H^*(B\GG_m) \simeq H^*(C) \otimes \QQ[z]$.

We want to perform a similar computation at the level of motives. First we need to understand vector bundle stacks.

\begin{defn}\label{def vector bundle stack}
Let $\fX$ be a stack and $d: \cE_0 \ra \cE_1$ be a homomorphism of vector bundles over $\fX$. Then the associated vector bundle stack is
\[ [\cE_1/\cE_0] \ra \fX. \]
\end{defn}

A formula for the class of vector bundle stacks in a dimensional completion of the Grothendieck ring of varieties is proved in \cite[Lemma 3.3]{GPHS}. For an analogous formula for Voevodsky motives, we give the following simple proof.

\begin{prop}\label{prop motive vb stack}
For a vector bundle stack $\cE\ra \fX$, we have $M(\cE) \simeq M(\fX)$.
\end{prop}
\begin{proof}
We claim that in $\DM(\fX,\QQ)$, we have $M(\cE \ra \fX)\simeq \QQ_\fX$. To prove this, it suffices to check this isomorphism holds after pulling to all points of $ x \in \fX$; this is true for motives on schemes by \cite[Proposition 3.24]{ayoub_etale}, and then follows for motives of stacks because the family of functors $u^{*}:\DM(\fX,\QQ)\to \DM(X,\QQ)$ with $u:X\to \fX$ varying across smooth morphisms from a scheme $X$ is jointly conservative \cite[Eq.\ (A.4)]{khan-stacks}. Since $\cE\ra \fX$ is smooth, relative purity and base change imply that $x^{*}M(\cE\ra\fX)\simeq M(\cE_{x}\ra \Spec(\kappa(x)))$. However, 
\[ \cE_x \simeq [V_1/V_0] \simeq \AA^r \times B\GG_a^s \]
where $r = \dim(V_1) - \rk(d)$ and $s = \dim(\ker(d))$. The motive of $\AA^{1}$ is trivial, as is the motive of $B\GG_{a}$ because it coincides with the motive of the simplicial classifying space of $\GG_{a}$ which is levelwise $\AA^{1}$-contractible. Consequently, $M(\cE_x) \simeq \QQ_{\kappa(x)}$.
\end{proof}

We can now prove the first statement in Theorem \ref{main thm}.

\begin{thm}\label{thm torsion sheaves}
Assume that $C(k) \neq \emptyset$; then in $\DM(k,\QQ)$, we have
\[M(\Coh_{0,d}) \simeq  \Sym^d(M(\Coh_{0,1})) \simeq \Sym^d(M(C \times B\GG_m)). \]
\end{thm}
\begin{proof}
By applying Theorem \ref{thm small maps stacks} to the forgetful map $\TCoh_{0,d}  \ra \Coh_{0,d}$, we have
\[ M(\Coh_{0,d}) \simeq M(\TCoh_{0,d})^{S_d}. \]
We apply Proposition \ref{prop motive vb stack} to  the associated graded map $\gr : \TCoh_{0,d}  \ra \prod_{i=1}^d \Coh_{0,1}$ to conclude
\[ M(\Coh_{0,d}) \simeq M(\TCoh_{0,d})^{S_d} \simeq M(\Pi_{i=1}^d \Coh_{0,1})^{S_d} \simeq \Sym^d(M(\Coh_{0,1})). \]
Finally, we have $M(\Coh_{0,1}) \simeq M(C \times B\GG_m) \simeq \bigoplus_{i \geq 0} M(C)\{ i\}$ as $\Coh_{0,1} \ra C$ is a trivial $\GG_m$-gerbe, because $C(k)\neq \emptyset$.
\end{proof}

\section{The stack of coherent sheaves}\label{sec:coh}
 
As in the cohomological description of Heinloth \cite{Heinloth_coh}, we use the torsion stratification on $\Coh_{n,d}$ to compute its motive. The key difference is that in the cohomological description the splitting of the associated Gysin long exact sequences uses the usual Atiyah--Bott method \cite{AB}, whereas in the motivic setting, to prove the corresponding Gysin distinguished triangles split, we use purity of the motives appearing. This purity can be expressed using the weight structure in the sense of Bondarko. However we do not really need to know anything about this weight structure, just about its heart. Here is the relevant definition.

\begin{defn}\label{def:pure}
The subcategory $\DM(k,\QQ)^{w=0}$ of \emph{pure} (or \emph{weight zero}) objects in $\DM(k,\QQ)$ is the full subcategory of $\DM(k,\QQ)$ containing motives of smooth projective varieties and stable by coproducts and direct factors.
\end{defn}

As the notation suggests, $\DM(k,\QQ)^{w=0}$ is the heart of the Chow weight structure on $\DM(k,\QQ)$, see \cite[Theorem 2.1.1 II]{Bondarko-rel}.

\begin{lemma}\label{lemma:pure-split}
Let $M\to N\to P\to M[1]$ be a distinguished triangle in $\DM(k,\QQ)$. Assume that $M$ and $P$ are pure objects. Then the distinguished triangle splits and $N$ is also pure.
\end{lemma}  
\begin{proof}
  We have to prove that the morphism $P\to M[1]$ is $0$. This follows from the definition of pure objects and  the observation that, if $X$ and $Y$ are smooth projective varieties over $k$, with $Y$ assumed equidimensional of dimension $d$ without loss of generality, we have
  \[
\Hom_{\DM(k,\QQ)}(M(X),M(Y)[1])=\Hom_{\DM(k,\QQ)}(M(X\times_{k}Y),\QQ(d)[2d+1])=\CH^{d}(X\times_{k}Y,-1)\otimes\QQ=0
\]
by Poincaré duality and the representability of higher Chow groups in $\DM(k,\QQ)$.
\end{proof}
  
Since every coherent sheaf $E$ on $C$ has a torsion filtration $E_{\mathrm{tor}} \hookrightarrow E \twoheadrightarrow E_{\mathrm{free}}$, one obtains a stratification by the degree $e$ of the torsion piece:
\[ \Coh_{n,d} = \bigsqcup_{e \geq 0} \Coh_{n,d}^{\deg(\mathrm{tor}) =e} \]
with $\Coh_{n,d}^{\deg(\mathrm{tor}) =0} = \Bun_{n,d}$. Furthermore, we have associated graded maps
\[ \gr : \Coh_{n,d}^{\deg(\mathrm{tor}) =e} \ra \Coh_{0,e} \times \Bun_{n,d-e} \]
which are vector bundle stacks. 

We are now able to prove the final statement in Theorem \ref{main thm}.

\begin{thm}\label{thm coh positive rank}
Assume that $C(k) \neq \emptyset$. Then in $\DM(k,\QQ)$ for $n >0$, we have
\[ M(\Coh_{n,d}) \simeq M(\Jac(C)) \otimes M(B\GG_m) \otimes \bigotimes_{i=1}^{\infty} Z(C, \QQ\{i\}).\]
\end{thm}
\begin{proof}
Since the graded map on each stratum in the torsion stratification is a vector bundle stack, we have
\[ M( \Coh_{n,d}^{\deg(\mathrm{tor}) =e} ) \simeq M( \Coh_{0,e}) \otimes M( \Bun_{n,d-e})\]
by Proposition \ref{prop motive vb stack}. Associated to the smooth closed pair $\Coh_{n,d}^{\deg(\mathrm{tor}) \leq e} \hookrightarrow \Coh_{n,d}^{\deg(\mathrm{tor}) =e}$ of codimension $ne$, there is a Gysin distinguished triangle
\[ M(\Coh_{n,d}^{\deg(\mathrm{tor}) <e}) \ra M(\Coh_{n,d}^{\deg(\mathrm{tor}) \leq e}) \ra M(\Coh_{n,d}^{\deg(\mathrm{tor}) =e})\{ ne \} \stackrel{+1}{\ra}\]
which we claim splits inductively. Indeed, as $M(\Coh_{n,d}^{\deg(\mathrm{tor}) =e})\{ ne \}$ is pure by the first two parts of Theorem \ref{main thm}, one can inductively show that $M(\Coh_{n,d}^{\deg(\mathrm{tor}) <e})$ is pure and this sequence splits by Lemma \ref{lemma:pure-split}. Consequently, we have
 \begin{align*}
M(\Coh_{n,d}) \: & \simeq \bigoplus_{e \geq 0} M( \Coh_{n,d}^{\deg(\mathrm{tor}) =e} )\{ ne \} \simeq \bigoplus_{e \geq 0} M( \Bun_{n,d-e})  \otimes M(\Coh_{0,e})\{ne \}
 \\
&\simeq M(\Bun_{n,d}) \otimes \left( \bigoplus_{e \geq 0} \Sym^e(M(C)\{n\} \otimes B\GG_m) \right) \\
&\simeq M(\Bun_{n,d}) \otimes \bigotimes_{k \geq n} Z(C,\QQ\{k \}) \\
&\simeq M(\Jac(C)) \otimes M(B\GG_m) \otimes \bigotimes_{i=1}^{\infty} Z(C, \QQ\{i\}),
\end{align*}
by the first two parts of Theorem \ref{main thm} and the fact that $ M(\Bun_{n,d})$ is independent of $d$.
\end{proof}

\bibliographystyle{amsplain}
\bibliography{references}

\providecommand{\bysame}{\leavevmode\hbox to3em{\hrulefill}\thinspace}
\providecommand{\MR}{\relax\ifhmode\unskip\space\fi MR }
\providecommand{\MRhref}[2]{%
  \href{http://www.ams.org/mathscinet-getitem?mr=#1}{#2}
}
\providecommand{\href}[2]{#2}
\begin{thebibliography}{10}

\bibitem{AB}
M.~F. Atiyah and R.~Bott, \emph{The {Y}ang-{M}ills equations over {R}iemann
  surfaces}, Philos. Trans. R. Soc. Lond., Ser. A \textbf{308} (1983),
  523--615.

\bibitem{ayoub_etale}
J.~Ayoub, \emph{La r\'ealisation \'etale et les op\'erations de
  {G}rothendieck}, Ann. Sci. \'Ec. Norm. Sup\'er. (4) \textbf{47} (2014),
  no.~1, 1--145.

\bibitem{Bondarko-rel}
M.~V. Bondarko, \emph{Weights for relative motives: relation with mixed
  complexes of sheaves}, Int. Math. Res. Not. IMRN (2014), no.~17, 4715--4767.

\bibitem{CD_int}
D.-C. Cisinski and F.~D\'{e}glise, \emph{Integral mixed motives in equal
  characteristic}, Doc. Math. (2015), no.~Extra vol.: Alexander S. Merkurjev's
  sixtieth birthday, 145--194.

\bibitem{dCM_small}
M.~de~Cataldo and L.~Migliorini, \emph{The {C}how motive of semismall
  resolutions}, Math. Res. Lett. \textbf{11} (2004), no.~2-3, 151--170.

\bibitem{GPHS}
O.~Garc{\'i}a-Prada, J.~Heinloth, and A.~Schmitt, \emph{On the motives of
  moduli of chains and {H}iggs bundles}, J. Eur. Math. Soc. (JEMS) \textbf{16}
  (2014), no.~12, 2617--2668.

\bibitem{Heinloth_coh}
J.~Heinloth, \emph{Cohomology of the moduli stack of coherent sheaves on a
  curve}, Geometry and arithmetic, EMS Ser. Congr. Rep., Eur. Math. Soc.,
  Z\"{u}rich, 2012, pp.~165--171.

\bibitem{Heinloth_LaumonBDay}
\bysame, \emph{A conjecture of {H}ausel on the moduli space of {H}iggs bundles
  on a curve}, Ast\'erisque (2015), no.~370, 157--175.

\bibitem{HPL}
V.~Hoskins and S.~Pepin~Lehalleur, \emph{On the {V}oevodsky motive of the
  moduli stack of vector bundles over a curve},  (2020), to appear in the
  Quartely Journal of Mathematics.

\bibitem{HPL_formula}
V.~Hoskins and S.~Pepin~Lehalleur, \emph{A formula for the {V}oevodsky motive
  of the moduli stack of vector bundles on a curve}, Geom. Topol. \textbf{25}
  (2021), no.~7, 3555--3589.

\bibitem{HPL_mirrorsym}
\bysame, \emph{Motivic mirror symmetry for {H}iggs bundles}, arXiv preprint:
  2205.15393, 2022.

\bibitem{khan-stacks}
A.~A. Khan, \emph{Virtual fundamental classes of derived stacks {I}}, 2019,
  arXiv preprint: 1909.01332.

\bibitem{Laumon}
G.~Laumon, \emph{Correspondance de {L}anglands g\'eom\'etrique pour les corps
  de fonctions}, Duke Math. J. \textbf{54} (1987), no.~2, 309--359.

\bibitem{HTT}
J.~Lurie, \emph{Higher topos theory}, Annals of Mathematics Studies, vol. 170,
  Princeton University Press, Princeton, NJ, 2009.

\bibitem{HA}
\bysame, \emph{Higher algebra},  (2017).

\bibitem{kerodon}
J.~Lurie, \emph{Kerodon}, \url{https://kerodon.net}, 2018.

\bibitem{MVW}
C.~Mazza, V.~Voevodsky, and C.~Weibel, \emph{Lecture notes on motivic
  cohomology}, Clay Mathematics Monographs, vol.~2, American Mathematical
  Society, Providence, RI; Clay Mathematics Institute, Cambridge, MA, 2006.

\bibitem{SSala}
F.~Sala and O.~Schiffmann, \emph{Cohomological {H}all algebra of {H}iggs
  sheaves on a curve}, Algebr. Geom. \textbf{7} (2020), no.~3, 346--376.

\bibitem{Schiffmann_survey}
O.~Schiffmann, \emph{Kac polynomials and {L}ie algebras associated to quivers
  and curves}, Proceedings of the {I}nternational {C}ongress of
  {M}athematicians---{R}io de {J}aneiro 2018. {V}ol. {II}. {I}nvited lectures,
  World Sci. Publ., Hackensack, NJ, 2018, pp.~1393--1424.

\bibitem{SV_preprint}
O.~Schiffmann and E.~Vasserot, \emph{The cohomological {H}all algebra of the
  stack of coherent sheaves on a smooth projective curve}, In preparation.

\end{thebibliography}

\medskip \medskip

\noindent{Radboud University, IMAPP, PO Box 9010,
6525 AJ Nijmegen, Netherlands} 

\medskip \noindent{\texttt{v.hoskins@math.ru.nl, simon.pepin.lehalleur@gmail.com}}

\end{document}